\documentclass{amsart}
\usepackage{mathrsfs,amssymb,mathrsfs, multirow,setspace}
\usepackage[a4paper, total={7in, 9in}]{geometry}
\usepackage{enumerate}
\theoremstyle{plain}
\usepackage{graphicx}
\usepackage{mathtools}

\newtheorem{theorem}{Theorem}[section]
\newtheorem{lemma}[theorem]{Lemma}
\newtheorem{proposition}[theorem]{Proposition}

\theoremstyle{definition}

\newtheorem{definition}[theorem]{Definition}

\theoremstyle{remark}
\newtheorem{remark}[theorem]{Remark}

\input xy
\xyoption{all}

\begin{document}

	\title[Characterization of rings with planar, toroidal or projective planar prime ideal sum graphs]{Characterization of rings with planar, toroidal or projective planar prime ideal sum graphs}
	
	\author[Praveen Mathil, Barkha Baloda, Jitender Kumar, A. Somasundaram]{Praveen Mathil$^{\dagger}$, Barkha Baloda$^{\dagger}$, Jitender Kumar$^{\dagger, *}$, A. Somasundaram$^{\mathsection}$}

\begin{abstract}
Let $R$ be a commutative ring with unity. The prime ideal sum graph $\text{PIS}(R)$ of the ring $R$ is the simple undirected graph whose vertex set is the set of all nonzero proper ideals of $R$ and two distinct vertices $I$ and $J$ are adjacent if and only if $I + J$ is a prime ideal of $R$. In this paper, we study some interplay between algebraic properties of rings and graph-theoretic properties of their prime ideal sum graphs. In this connection, we classify non-local commutative Artinian rings $R$ such that $\text{PIS}(R)$ is of crosscap at most two. We prove that there does not exist a non-local commutative Artinian ring whose prime ideal sum graph is projective planar. Further, we classify non-local commutative Artinian rings of genus one prime ideal sum graphs.
\end{abstract}

\keywords{Non-local ring, ideals, genus and crosscap of a graph, prime ideal.\\
*  Corresponding author \\
$^{\dagger}$ Department of Mathematics, Birla Institute of Technology and Science Pilani, Pilani-333031 (Rajasthan), India \\
$^{\mathsection}$ Department of General Sciences, Birla Institute of Technology and Science, Pilani, Dubai Campus, Dubai, UAE \\ Email address: maithilpraveen@gmail.com, barkha0026@gmail.com, jitenderarora09@gmail.com, asomasundaram@dubai.bits-pilani.ac.in}

\maketitle

\section{Introduction}
The investigation of algebraic structures through graph-theoretic properties of the associated graphs has become a fascinating research subject over the past few decades. Numerous graphs attached with ring structures have been studied in the literature (see, \cite{afkhami2012generalized, akbari2009total, anderson2008zero, anderson2008total, anderson1999zero, beck1988coloring, behboodi2011annihilating, biswas2022subgraph, maimani2008comaximal}). Because of the crucial role of ideals in the theory of rings, many authors have explored the graphs attached to the ideals of rings also, for example, inclusion ideal graph \cite{akbari2015inclusion}, intersection graphs of ideals \cite{chakrabarty2009intersection}, ideal-relation graph \cite{ma2016automorphism}, co-maximal ideal graph \cite{ye2012co}, prime ideal sum graph \cite{saha2022prime} etc. Topological graph theory is principally related to the embedding of a graph on a surface without edge crossing. Its applications lie in electronic printing circuits where the purpose is to embed a circuit, that is, the graph on a circuit board (the surface) without two connections crossing each other, resulting in a short circuit. To determine the genus and crosscap of a graph is a fundamental but highly complex problem. Indeed, it is NP-complete. Many authors have investigated the problem of finding the genus of zero divisor graphs of rings in \cite{akbari2003zero, asir2020classification, belshoff2007planar,  wang2006zero, wickham2008classification, wickham2009rings}. Genus and crosscap of the total graph of the ring were investigated in \cite{asir2019classification, khashyarmanesh2013projective, maimani2012rings,chelvam2013genus}.
 Asir \emph{et al.} \cite{asir2018classification} determined all isomorphic classes of commutative rings whose ideal based total graph has genus at most two. Pucanovi\'{c} \emph{et al.} \cite{pucanovic2014toroidality} classified the planar and toroidal graphs that are intersection ideal graphs of Artinian commutative rings. In \cite{pucanovic2014genus}, all the graphs of genus two that are intersection graphs of ideals of some commutative rings are characterized. Ramanathan \cite{ramanathan2021projective} determined all Artinian commutative rings whose intersection ideal graphs have crosscap at most two. Work related to the embedding of graphs associated with other algebraic structures on a surface can be found in \cite{afkhami2015planar, anitha2020characterization, das2015genus, das2016genus, kalaimurugan2021genus, kalaimurugan2022genus, khorsandi2022nonorientable, ma2020finite, rilwan2020genus, selvakumar2022genus, tamizh2020genus}.

 Recently, Saha \emph{et al.} \cite{saha2022prime} introduced and studied the prime ideal sum graph of a commutative ring. The \emph{prime ideal sum graph} $\text{PIS}(R)$ of the ring $R$ is the simple undirected graph whose vertex set is the set of all nonzero proper ideals of $R$ and two distinct vertices $I$ and $J$ are adjacent if and only if $I + J$ is a prime ideal of $R$. Authors of \cite{saha2022prime} studied the interplay between graph-theoretic properties of $\text{PIS}(R)$ and algebraic properties of ring $R$. In this connection, they investigated the clique number, the chromatic number and the domination number of prime ideal sum graph $\text{PIS}(R)$. The purpose of this article is to investigate the prime ideal sum graph $\text{PIS}(R)$ to a greater extent. In this connection, we discuss the question of embedding of $\text{PIS}(R)$ on various surfaces without edge crossing. This paper aims to characterize all commutative non-local Artinian rings for which $\text{PIS}(R)$ has crosscap at most two. We also characterise all the non-local commutative Artinian rings whose prime ideal sum graph is toroidal. Moreover, we classify all the non-local commutative Artinian rings for which  $\text{PIS}(R)$ is planar and outerplanar, respectively. The paper is arranged as follows. Section 2 comprises basic definitions and necessary results. In Section 3, we classify all the non-local commutative Artinian rings $R$ for which $\text{PIS}(R)$ has genus one. Also, we determine all the non-local commutative Artinian rings for which $\text{PIS}(R)$ has crosscap at most two. 

 \section{preliminaries}
A \emph{graph} $\Gamma$ is a pair  $(V(\Gamma), E(\Gamma))$, where $V(\Gamma)$ and $E(\Gamma)$ are the set of vertices and edges of $\Gamma$, respectively. Two distinct vertices $u_1$ and $u_2$ are $\mathit{adjacent}$, denoted by $u_1 \sim u_2$ (or $(u_1, \ u_2)$), if there is an edge between $u_1$ and $u_2$. Otherwise, we write as $u_1 \nsim u_2$. If $X \subseteq V(\Gamma)$ then the subgraph $\Gamma(X)$ induced by $X$ is the graph with vertex set $X$ and two vertices of $\Gamma(X)$ are adjacent if and only if they are adjacent in $\Gamma$. For other basic graph theoretic definitions and concepts, we refer the reader to \cite{westgraph, white1985graphs}.
A graph $\Gamma$ is \emph{outerplanar} if it can be embedded in the plane such that all vertices lie on the outer face of $\Gamma$. In a graph $\Gamma$, the \emph{subdivision} of an edge $(u,v)$ is the deletion of $(u,v)$ from $\Gamma$ and the addition of two edges $(u,w)$ and $(w,v)$ along with a new vertex $w$. A graph obtained from $\Gamma$ by a sequence of edge subdivision is called a subdivision of $\Gamma$. Two graphs are said to be \emph{homeomorphic} if both can be obtained from the same graph by subdivisions of edges. A graph $\Gamma$ is \emph{planar} if it can be drawn on a plane without edge crossing. It is well known that every outerplanar graph is a planar graph. The following results will be useful for later use.

\begin{theorem}\label{outerplanar criteria}\cite{westgraph}
A graph $\Gamma$ is outerplanar if and only if it does not contain a subdivision of $K_4$ or $K_{2,3}$.
\end{theorem}
\begin{theorem}\label{planar criteria}\cite{westgraph}
A graph $\Gamma$ is planar if and only if it does not contain a subdivision of $K_5$ or $K_{3,3}$.
\end{theorem}

A compact connected topological space such that each point has a neighbourhood homeomorphic to an open disc is called a surface. For a non-negative integer $g$, let $\mathbb{S}_{g}$ be the orientable surface with $g$ handles. The genus $g(\Gamma)$ of a graph $\Gamma$ is the minimum integer $g$ such that the graph can be embedded in $\mathbb{S}_{g}$, i.e. the graph $\Gamma$ can be drawn into a surface $\mathbb{S}_{g}$ with no edge crossing. Note that the graphs having genus $0$ are planar and the graphs having genus one are toroidal. The following results are useful in the sequel.

\begin{proposition}{\cite[Ringel and Youngs]{white1985graphs}}
\label{genus}
 Let $m, n$ be positive integers. 
 \begin{enumerate}
     \item[{\rm(i)}]If $n \ge 3$, then $g(K_n) = \left\lceil \frac{(n-3)(n-4)}{12} \right\rceil$.
     \item[{\rm(ii)}]If $m, n\geq 2$, then $g(K_{m,n}) =  \left\lceil \frac{(m-2)(n-2)}{4}  \right\rceil$.
 \end{enumerate}
 \end{proposition}
 
 \begin{lemma}{\cite[Theorem 5.14]{white1985graphs}}
 \label{eulerformulagenus}
 Let $\Gamma$ be a connected graph with a 2-cell embedding in $\mathbb{S}_{g}$. Then $v - e + f = 2 - 2g$, where $v, e$ and $f$ are the number of vertices, edges and faces embedded in $\mathbb{S}_{g}$, respectively and $g$ is the genus of the graph $\Gamma$. 
\end{lemma}
\begin{lemma}\cite{white2001graphs}
\label{genusofblocks}
The genus of a connected graph $\Gamma$ is the sum of the genera of its blocks.
\end{lemma}

Let $\mathbb{N}_{k}$ denote the non-orientable surface formed by the connected sum of $k$ projective planes, that is, $\mathbb{N}_{k}$ is a non-orientable surface with $k$ crosscap. The \emph{crosscap} of a graph $\Gamma$, denoted by $cr(\Gamma)$, is the minimum non-negative integer $k$ such that $\Gamma$ can be embedded in $\mathbb{N}_{k}$. For instance, a graph $\Gamma$ is planar if $cr(\Gamma) = 0$ and the  $\Gamma$ is projective if $cr(\Gamma) = 1$. The following results are useful to obtain the crosscap of a graph.
\begin{proposition}\label{crosscap}{\cite[Ringel and Youngs]{mohar2001graphs}}
 Let $m, n$ be positive integers. Then\\
$ {\rm(i)}~~ cr(K_n) =
 \begin{cases} 
      \left\lceil \frac{(n-3)(n-4)}{6} \right\rceil&  \textit{if}~~ n\geq 3 \\
    3 & \textit{if}~~ n =7 
   \end{cases}$\\
${\rm(ii)}~~ cr(K_{m,n}) =  \left\lceil \frac{(m-2)(n-2)}{2}  \right\rceil~~ \textit{if}~~ m, n\geq 2$
\end{proposition}

\begin{lemma}\label{eulerformulacrosscap}{\cite[Lemma 3.1.4]{mohar2001graphs}}
Let $\phi : \Gamma \rightarrow \mathbb{N}_{k}$ be a $2$-cell embedding of a connected graph $\Gamma$ to the non-orientable surface $\mathbb{N}_{k}$. Then $v - e + f = 2 - k$, where $v, e$ and $f$ are the number of vertices, edges and faces of $\phi(\Gamma)$ respectively, and $k$ is the crosscap of $\mathbb{N}_{k}$.
\end{lemma}

\begin{definition}\cite{white2001graphs}
 A graph $\Gamma$ is orientably simple if $\mu(\Gamma) \neq 2-cr(\Gamma)$, where $\mu(\Gamma) = \rm{max} \{2- 2$$g$$ (\Gamma), 2-cr(\Gamma)\}$.

\end{definition}

\begin{lemma}\cite{white2001graphs}\label{crosscapofblocks}
Let $\Gamma$ be a graph with blocks $\Gamma_1, \Gamma_2, \cdots, \Gamma_k$. Then 
\begin{center}
$cr(\Gamma) =$
$\begin{cases} 
1-k+ \sum \limits_{i=1}^{k} cr(\Gamma_i),  & ~~\textit{if}~~ \Gamma~~ \textit{is orientably simple}\\
2k - \sum \limits_{i=1}^{k} \mu (\Gamma_i), & ~~\textit{otherwise.}
\end{cases}$
\end{center}
\end{lemma}

We use the following remark frequently in this paper.
\begin{remark}\label{triangularface}
For a simple graph $\Gamma$, we have $2e \geq 3f$. 
\end{remark}

  A ring $R$ is called \emph{local} if it has a unique maximal ideal $\mathcal{M}$ and it is abbreviated by $(R, \mathcal{M})$. For an ideal $I$ of $R$, the smallest positive integer $n$ such that $I^n = 0$ is called the \emph{nilpotency index} $\eta(I)$ of the ideal $I$. Let $R$ be a non-local commutative Artinian ring. 
 By the structural theorem (see \cite{atiyah1994introduction}),  $R$ is uniquely (up to isomorphism) a finite direct product of local rings $R_i$ that is $R \cong R_1 \times R_2 \times \cdots \times R_n$, where $n \geq 2$. Note that for any commutative Artinian ring, every prime ideal is a maximal ideal (see {\cite[Proposition 8.1]{atiyah1994introduction}}). Hence, we use the maximal ideal for the adjacency between the two vertices of the prime ideal sum graph. The following remark specifies maximal ideals in a non-local Artinian commutative ring, where $\textnormal{Max}(R)$ denotes the set of all maximal ideals of $R$. The set of non-zero proper ideals of $R$ is denoted by $\mathcal{I}^*(R)$.  Throughout the paper, $F_i$ denotes a  field. For other basic definitions of ring theory, we refer the reader to \cite{atiyah1994introduction}.

 \begin{remark}
Let $R \cong R_1 \times R_2 \times \cdots \times R_n$ $(n \geq 2)$ be a non-local commutative ring, where each $R_i$ is a local ring with maximal ideal $\mathcal{M}_i$. Then $\textnormal{Max}(R) = \{ J_1, J_2, \ldots, J_n\}$, where $J_i =  R_1 \times R_2 \times \cdots \times R_{i-1} \times \mathcal{M}_i \times R_{i+1} \times \cdots \times  R_n$.
\end{remark}
 

\section{Embedding of $\text{PIS}(R)$ on surfaces}
In this section, we study the embedding of the prime ideal sum graph $\text{PIS}(R)$ on a surface without edge crossing. We begin with the investigation of an embedding of $\text{PIS}(R)$ on a plane.

\subsection{Planarity of $\text{PIS}(R)$}
In this subsection, we classify all the non-local commutative Artinian rings with unity for which the graph $\text{PIS}(R)$ is planar and outerplanar, respectively.

\begin{theorem}\label{Planar_primeidealsum}
Let $R \cong R_1 \times R_2 \times \cdots \times R_n$ $(n \geq 2)$ be a non-local commutative ring, where each $R_i$ is a local ring with maximal ideal $\mathcal{M}_i$ and let $F_1$, $F_2$, $F_3$ be fields. Then the graph $\textnormal{PIS}(R)$ is planar if and only if one of the following holds:
 \begin{enumerate}[\rm(i)]
    \item $R \cong  F_1 \times F_2 \times F_3$. 
    \item $R \cong  F_1 \times F_2$.
    \item $R \cong  R_1 \times R_2$ such that both $R_1$ and $R_2$ are principal ideal rings with $\eta(\mathcal{M}_1) = \eta(\mathcal{M}_2) =2$.
    \item $R \cong  F_1 \times R_2$, where the local ring $R_2$ is principal ideal ring.
    
    \end{enumerate}
\end{theorem}

\begin{proof}
First suppose that the graph $\text{PIS}(R)$ is planar and $R \cong R_1 \times R_2 \times \cdots \times R_n$, where $n \ge 5$. Then the set $X = \{\mathcal{M}_1 \times R_2 \times \cdots \times R_n, \ \mathcal{M}_1 \times \langle 0 \rangle \times R_3 \times \cdots \times R_n, \ \mathcal{M}_1 \times R_2 \times \langle 0 \rangle \times R_4 \times \cdots \times R_n, \ \mathcal{M}_1 \times R_2 \times R_3 \times \langle 0 \rangle \times R_5 \times \cdots \times R_n, \ \mathcal{M}_1 \times R_2 \times R_3 \times R_4 \times \langle 0 \rangle \times R_6 \times \cdots \times R_n\}$ induces a subgraph isomorphic to $K_5$, a contradiction. Therefore, $n \in \{2, 3, 4\}$. If $R \cong R_1 \times R_2 \times R_3 \times R_4$, then by Figure \ref{subdivision of K33 fig1} and Theorem \ref{planar criteria}, the graph $\text{PIS}(R)$ is not planar, a contradiction. Consequently, $n \in \{2, 3\}$. 
\begin{figure}[h!]
\centering
\includegraphics[width=0.6 \textwidth]{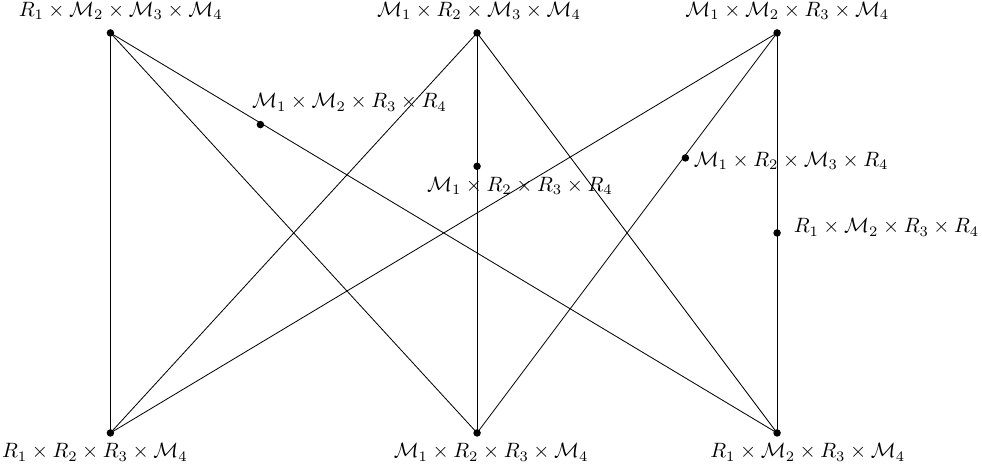}
\caption{A subgraph of $\text{PIS}(R_1 \times R_2 \times R_3 \times R_4)$ homeomorphic to $K_{3,3}$}
\label{subdivision of K33 fig1}
\end{figure}

 Now assume that $R\cong R_1 \times R_2 \times R_3$ such that one of $R_i ( 1 \le i \le 3)$ is not a field. Without loss of generality, we assume that $R_1$ is not a field. By Figure \ref{subdivision of K5 fig1}, note that $\text{PIS}(R)$ contains a subgraph homeomorphic to $K_5$, which is not possible. Thus, $R \cong F_1\times F_2 \times F_3$.
\begin{figure}[h!]
\centering
\includegraphics[width=0.4 \textwidth]{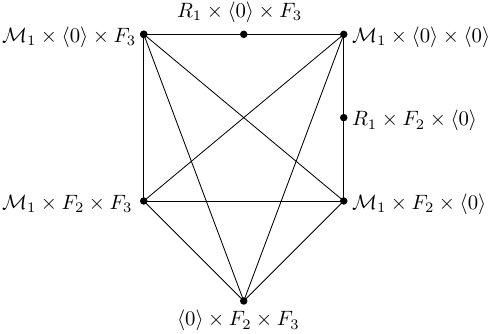}
\caption{A subgraph of $\text{PIS}(R_1 \times F_2 \times F_3)$ homeomorphic to $K_{5}$}
\label{subdivision of K5 fig1}
\end{figure}
We may now suppose that $R \cong R_1 \times R_2$. First note that both $R_1$ and $R_2$ are principal rings. On contrary, if one of them is not principal say $R_1$. Then $R_1$ has atleast two nontrivial ideals $J_1$ and $J_2$ diffrent from $\mathcal{M}_1$ such that $J_1 + J_2 = \mathcal{M}_1$. By Figure \ref{subdivision of K33 fig2}, note that $\text{PIS}(R_1 \times R_2)$ contains a subgraph homeomorphic to $K_{3,3}$, which is not possible.
\begin{figure}[h!]
\centering
\includegraphics[width=0.4 \textwidth]{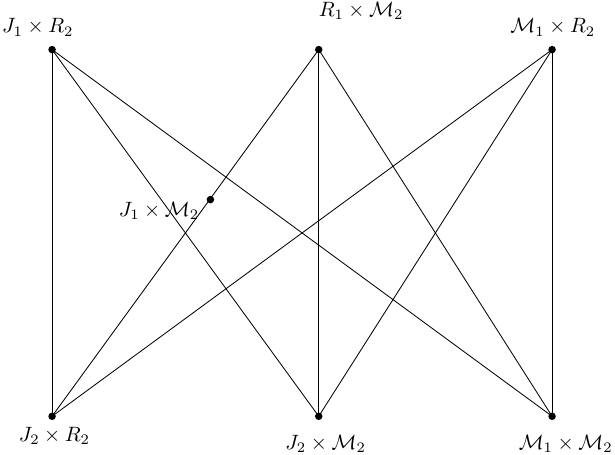}
\caption{Subgraph of $\text{PIS}(R_1 \times R_2)$, where $R_1$ is not a principal ring}
\label{subdivision of K33 fig2}
\end{figure}
Consequently, $R \cong R_1 \times R_2$ such that both the rings $R_1$ and $R_2$ are principal. If at least one of $R_i$ $(1 \le i \le 2)$ is a field, then $R$ is isomorphic to $F_1 \times F_2$ or $F_1 \times R_2$, where $R_2$ is a principal ideal ring. Now let $R \cong R_1 \times R_2$ such that both $R_1$ and $R_2$ are principal ideal rings which are not fields. If $ \eta(\mathcal{M}_1) \ge 3$ and $ \eta(\mathcal{M}_2) \ge 2$, then there exists a non-zero proper ideal $\mathcal{M}_1^2(\neq \mathcal{M}_1)$ of $R_1$. By Figure \ref{subdivision of K33 fig3}, the graph $\text{PIS}(R_1 \times R_2)$ contains a subgraph homeomorphic to $K_{3,3}$, a contradiction. Consequently, for $R \cong R_1 \times R_2$, where both $R_1$ and $R_2$ are not fields, we have $\eta(\mathcal{M}_1) = \eta(\mathcal{M}_2) =2$.
\begin{figure}[h!]
\centering
\includegraphics[width=0.4 \textwidth]{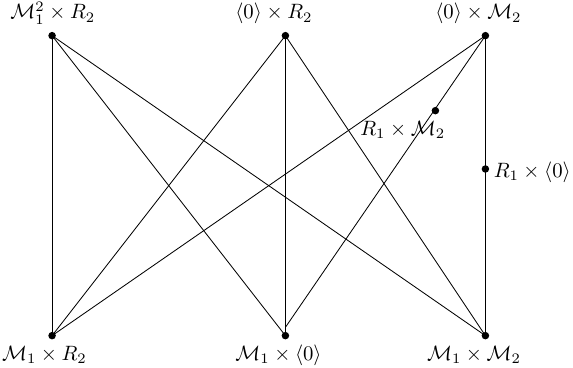}
\caption{Subgraph of $\text{PIS}(R_1 \times R_2)$, where $ \eta(\mathcal{M}_1) \ge 3$ and $ \eta(\mathcal{M}_1) \ge 2$}
\label{subdivision of K33 fig3}
\end{figure}

 The converse follows from Figures \ref{outerplanar_fig2}, \ref{outerplanar_fig3}, \ref{planar drawing fig1} and \ref{planar drawing fig2}.
 \begin{figure}[h!]
\centering
\includegraphics[width=0.4 \textwidth]{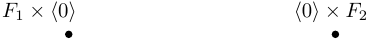}
\caption{The prime ideal sum graph of $F_1 \times F_2$}
\label{outerplanar_fig2}
\end{figure}
 \begin{figure}[h!]
\centering
\includegraphics[width=0.4 \textwidth]{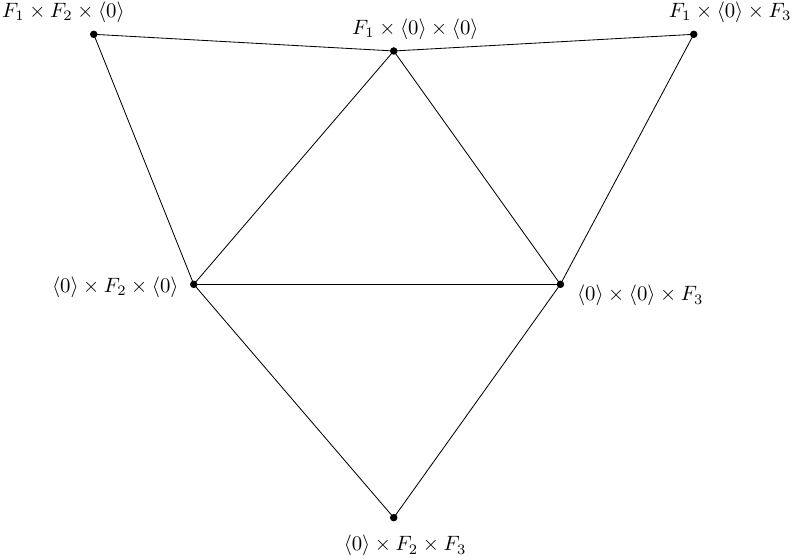}
\caption{ The graph $\text{PIS}(F_1 \times F_2 \times F_3)$}
\label{outerplanar_fig3}
\end{figure}
 \begin{figure}[h!]
\centering
\includegraphics[width=0.6 \textwidth]{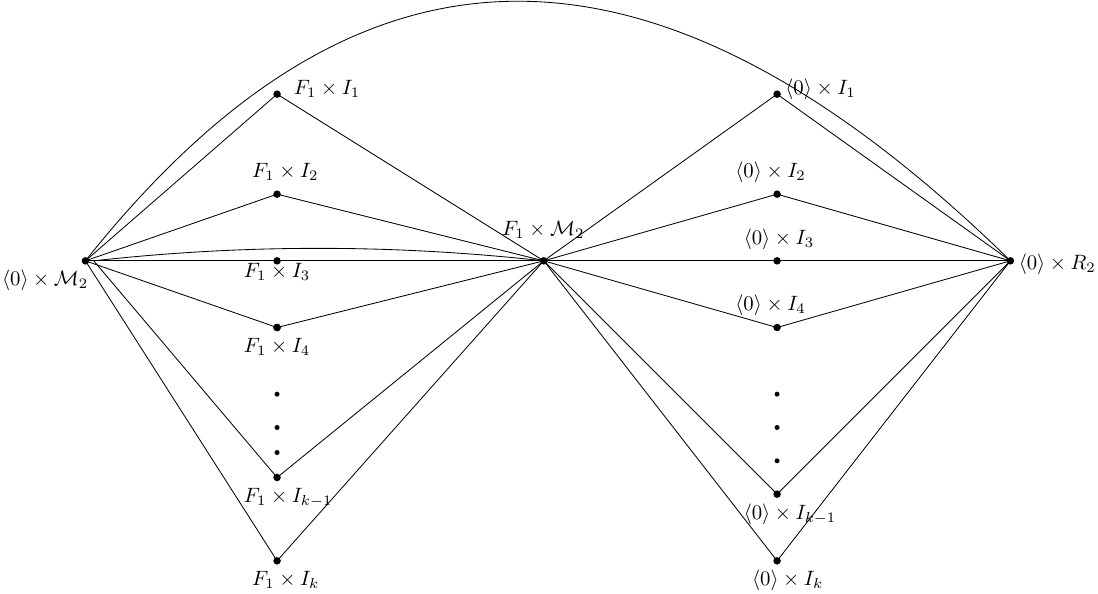}
\caption{ The graph $\text{PIS}(F_1 \times R_2)$, where $R_2$ is a principal ideal ring}
\label{planar drawing fig1}
\end{figure}
\begin{figure}[h!]
\centering
\includegraphics[width=0.4 \textwidth]{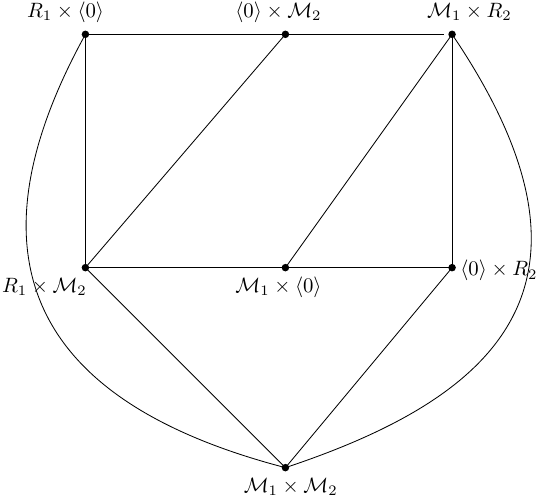}
\caption{The graph $\text{PIS}(R_1 \times R_2)$, where $\eta(\mathcal{M}_1) = \eta(\mathcal{M}_2) =2$}
\label{planar drawing fig2}
\end{figure}
\end{proof}

\begin{theorem}\label{outerplanar_primeidealsum}
Let $R \cong R_1 \times R_2 \times \cdots \times R_n$ $(n \geq 2)$ be a non-local commutative ring, where each $R_i$ is a local ring with maximal ideal $\mathcal{M}_i$. Then the graph $\textnormal{PIS}(R)$ is outerplanar if and only if $R$ is isomorphic to one of the following rings: $F_1 \times F_2 \times F_3$, $F_1 \times F_2$, $F_1 \times R_2$, where $F_1$, $F_2$, $F_3$ are fields and $R_2$ is a principal ideal ring with $\eta(\mathcal{M}_2) = 2$.
\end{theorem}

\begin{proof}
Suppose that $\text{PIS}(R)$ is an outerplanar graph. Since every outerplanar graph is planar, we get $R$ is isomorphic to the following rings: $F_1 \times F_2 \times F_3$, $F_1 \times F_2$, $F_1 \times R_2'$, $R_1 \times R_2$, where $R_2'$ is any principal ideal ring and $R_1$, $R_2$ are principal ideal rings with $\eta(\mathcal{M}_1) = \eta(\mathcal{M}_2) = 2$.

Let $R \cong R_1 \times R_2$, where $R_1$ and $R_2$ are principal rings. First suppose that both $R_1$ and $R_2$ are not fields together with $\mathcal{I}^*(R_1) = \{ \mathcal{M}_1 \}$ and $\mathcal{I}^*(R_2) = \{ \mathcal{M}_2 \}$. Consider $J_1 = \mathcal{M}_1 \times R_2 $, $J_2 = R_1 \times \mathcal{M}_2$, $J_3 = \langle 0 \rangle \times R_2$, $J_1' = \mathcal{M}_1 \times \mathcal{M}_2$, $J_2' = \mathcal{M}_1 \times \langle 0 \rangle$. Notice that the subgraph induced by the set $\{J_1, J_2, J_3, J_1', J_2' \}$ of vertices of $\text{PIS}(R)$ contains $K_{3,2}$ as a subgraph, a contradiction. Without loss of generality, we may now assume that $R_1$ is a field. Further, let $R_2$ is a principal ideal ring with $\eta(\mathcal{M}_2) \ge 3$. By Figure \ref{notouterplanar_fig1}, the graph $\text{PIS}(F_1 \times R_2)$ contains a subgraph homeomorphic to $K_{2,3}$, a contradiction. Therefore, either $R \cong F_1 \times F_2$ or $R \cong F_1 \times R_2$, where $R_2$ is a principal ideal ring with $\eta(\mathcal{M}_2) = 2$.
\begin{figure}[h!]
\centering
\includegraphics[width=0.3 \textwidth]{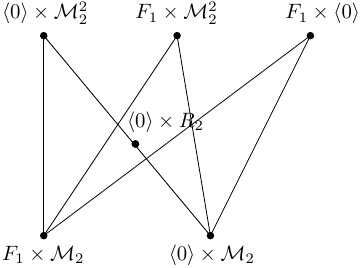}
\caption{Subgraph of $\text{PIS}(F_1 \times R_2)$, where $\eta(\mathcal{M}_2) \ge 3$}
\label{notouterplanar_fig1}
\end{figure}
The converse follows from Figures \ref{outerplanar_fig2}, \ref{outerplanar_fig3} and \ref{outerplanar_fig4}.
\begin{figure}[h!]
\centering
\includegraphics[width=0.3 \textwidth]{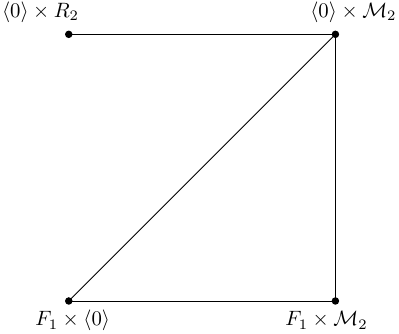}
\caption{The graph $\text{PIS}(F_1 \times R_2)$, where $\eta(\mathcal{M}_2) = 2$}
\label{outerplanar_fig4}
\end{figure}
 \end{proof}

\subsection{Genus of $\text{PIS}(R)$}
In this subsection, we classify all the non-local commutative Artinian rings with unity $R$ such that $\text{PIS}(R)$ has toroidal embedding. By $F_{123\ldots n}$, we mean $F_1 \times F_2 \times \cdots \times F_n$. We begin with the following lemma.

\begin{lemma}\label{genusgreaterthan2}
Let $R \cong F_1 \times F_2 \times F_3 \times F_4 \times F_5$, where each $F_i$ $(1 \le i \le 5)$ is a field. Then $g(\textnormal{PIS}(R)) \geq 3$.   
\end{lemma}

\begin{proof}
Let $R \cong F_1 \times F_2 \times F_3 \times F_4 \times F_5$. Consider the subset 
\[ X = \{F_{145}, F_{24}, F_{45}, F_{14}, F_{15}, F_{345}, F_{1245}, F_{13}, F_{135}, F_{12}, F_{134}, F_{1345}, F_{1235}, F_{234}, F_{35}, F_{245}, F_{235}\}\] 
of vertices of $\text{PIS}(R)$. Note that the subgraph induced by $X$ contains a subgraph homeomorphic to $K_{5,5}$ (see Figure \ref{subdivisionK_55fig1}). By Proposition \ref{genus}, we have $g(\text{PIS}(R)) \geq 3$. 
\begin{figure}[h!]
\centering
\includegraphics[width=0.6 \textwidth]{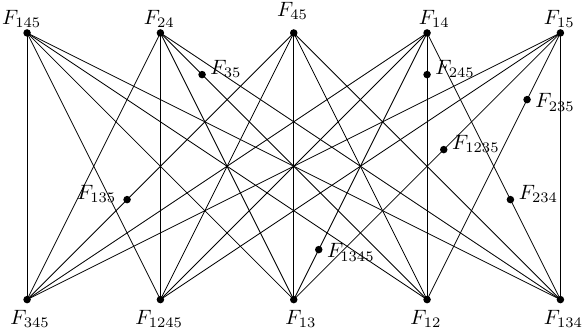}
\caption{Subgraph of $\text{PIS}(F_1 \times F_2 \times F_3 \times F_4 \times F_5)$}
\label{subdivisionK_55fig1}
\end{figure}
\end{proof}

\begin{theorem}\label{genus13product}
Let $R \cong R_1 \times R_2 \times \cdots \times R_n$ $(n \geq 3)$ be a non-local commutative ring, where each $R_i$ is a local ring with maximal ideal $\mathcal{M}_i$. Then $g(\textnormal{PIS}(R)) = 1$ if and only if $R \cong R_1 \times F_2 \times F_3$, where $R_1$ is a principal ideal ring with $\eta(\mathcal{M}_1) = 2$ and $F_2$, $F_3$ are fields.
\end{theorem}

\begin{proof}
First suppose that $g(\text{PIS}(R)) = 1$. By Lemma \ref{genusgreaterthan2}, we have $n \leq 4$. Let $R \cong R_1 \times R_2 \times R_3 \times R_4$. Assume that each $R_i$ ($1 \leq i \leq 4$) is a field. Notice that $v = 14$ and $e = 48$. By Lemma \ref{eulerformulagenus}, we obtain $f = 34$. It follows that $2e < 3f$, a contradiction. Consequently, $R \cong R_1 \times R_2 \times R_3$. If each $R_i$ is a field, then by Figure \ref{outerplanar_fig3}, $\text{PIS}(R)$ is a planar graph, which is not possible. Next, assume that $R_1$ and $R_2$ are not fields. Then by Figure \ref{K_54subdivisionR_1R_2F_3prime}, the graph $\text{PIS}(R_1 \times R_2 \times F_3 )$ contains a subgraph homeomorphic to $K_{5,4}$, a contradiction.
\begin{figure}[h!]
\centering
\includegraphics[width=0.6 \textwidth]{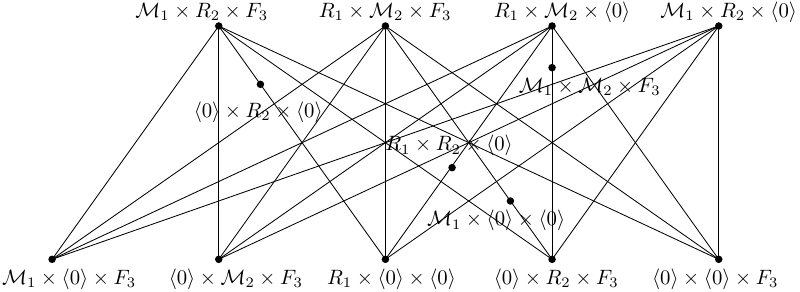}
\caption{Subgraph of $\text{PIS}(R_1 \times R_2 \times F_3 )$}
\label{K_54subdivisionR_1R_2F_3prime}
\end{figure}
It follows that $R_1$ is not a field but $R_2$ and $R_3$ are fields. Now assume that either $R_1$ is not a principal ideal ring or $R_1$ is a principal ideal ring with $\eta(\mathcal{M}_1) \ge 3$. Then $|\mathcal{I}^{*}(R_1)| \ge 2$. Let $K$ $(\neq \mathcal{M}_1)$ be a non-zero proper ideal of $R_1$. Then by Figure \ref{blockK_33fig1primeideal} and Lemma \ref{genusofblocks}, we get $g(\text{PIS}(R)) >1$, a contradiction. 
\begin{figure}[h!]
\centering
\includegraphics[width=0.6 \textwidth]{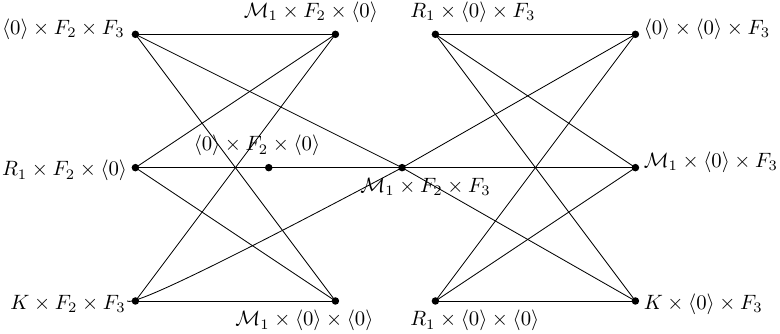}
\caption{Subgraph of $\text{PIS}(R_1 \times F_2 \times F_3)$, where $\mathcal{I}^{*}(R_1) = \{ K, \mathcal{M}_1 \}$}
\label{blockK_33fig1primeideal}
\end{figure}

Therefore, $R \cong R_1 \times F_2 \times F_3$, where $R_1$ is a principal ideal ring with $\eta(\mathcal{M}_1) = 2$. The converse follows from Figure \ref{genus1figure1primeideal}.
\begin{figure}[h!]
\centering
\includegraphics[width=0.6 \textwidth]{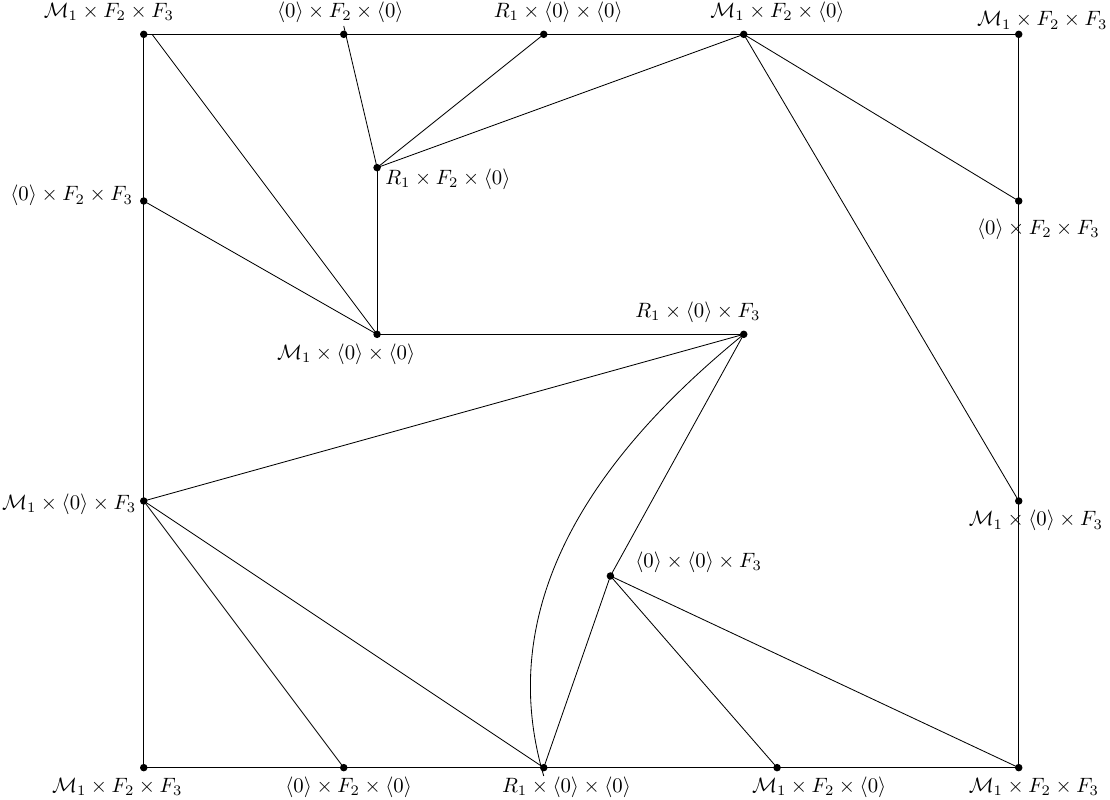}
\caption{Embedding of $\text{PIS}(R_1 \times F_2 \times F_3)$ in $\mathbb{S}_1$, where $\eta(\mathcal{M}_1) = 2$}
\label{genus1figure1primeideal}
\end{figure}
\end{proof}

\begin{theorem}\label{genusofR1R2}
  Let $R \cong R_1 \times R_2$, where $R_1$ and $R_2$ are local rings with maximal ideals $\mathcal{M}_1$ and $\mathcal{M}_2$, respectively. Then $g(\textnormal{PIS}(R)) = 1$ if and only if $R_1$ and $R_2$ are principal ideal rings with $\eta(\mathcal{M}_1) = 3$ and $\eta(\mathcal{M}_2) = 2$.
\end{theorem}

\begin{proof}
Suppose that $g(\text{PIS}(R)) = 1$. Assume that one of $R_i$ $(1 \le i \le 2)$ is not a principal ideal ring. Without loss of generality, let $R_1$ is not a principal ideal ring. Then $\mathcal{M}_1 = \langle x_1, x_2, \ldots, x_r \rangle$, where $r \ge 2$. Let $r=2$ i.e., $\mathcal{M}_1 = \langle x, y \rangle$ for some $x,y \in R_1$. Now let $\mathcal{I}^*(R_2) = \{\mathcal{M}_2\}$. Consider the vertices $I_1 = \langle x\rangle \times \langle 0 \rangle$, $I_2 = \langle x+y \rangle \times \langle 0 \rangle$, $I_3 = R_1 \times \mathcal{M}_2$, $I_4 = \langle x \rangle  \times R_2$, $I_5 = \langle y \rangle  \times R_2$, $I_6 = \langle x+y \rangle  \times R_2$, $I_7 = \mathcal{M}_1  \times R_2$, $I_8 = \mathcal{M}_1 \times \mathcal{M}_2$, $I_9 =  \langle x+y \rangle \times \mathcal{M}_2$ and $I_{10} =  \mathcal{M}_1 \times \langle 0 \rangle$. Note that the subgraph induced by these vertices contains a subgraph homeomorphic to $K_{5,4}$ (see Figure \ref{subdivision of K54 fig1}).
\begin{figure}[h!]
\centering
\includegraphics[width=0.6 \textwidth]{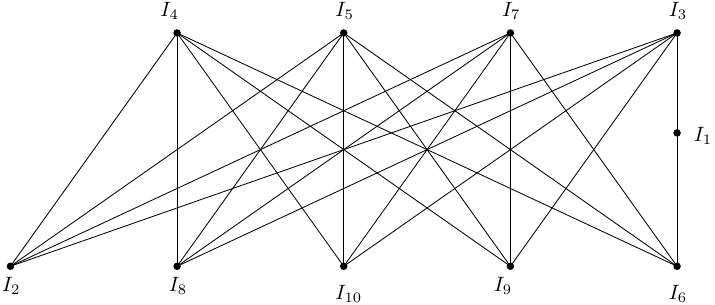}
\caption{Subgraph of $\text{PIS}(R_1 \times R_2)$, where $\mathcal{M}_1 = \langle x, y \rangle$ and $\mathcal{I}^*(R_2) = \{\mathcal{M}_2\}$}
\label{subdivision of K54 fig1}
\end{figure}
It follows that $g(\text{PIS}(R)) >1$, a contradiction. Next let, $R_2$ be a field. Let $x^2 \neq 0$. Then consider the ideals $J_1 = \langle x \rangle \times F_2$, $J_2 = \langle x^2, x +y \rangle \times F_2$, $J_3 = \langle x+y \rangle \times F_2$, $J_4 = \mathcal{M}_1 \times F_2$, $ J_5 = R_1 \times \langle 0 \rangle$, $J_1' = \langle y \rangle \times F_2$, $J_2' = \langle x^2, y \rangle \times \langle 0 \rangle$, $J_3' = \mathcal{M}_1 \times \langle 0 \rangle$, $J_4' = \langle y \rangle \times \langle 0 \rangle$ and $K_1 = \langle x \rangle \times \langle 0 \rangle$. Then the subgraph induced by these vertices contains a subgraph homeomorphic to $K_{5,4}$, where $\{ J_1, J_2, J_3, J_4, J_5 \}$ and $\{ J_1', J_2', J_3', J_4'\}$ are partition sets and $K_1$ join the vertices $J_5$ and $J_1'$. By Proposition \ref{genus}, we get a contradiction. Consequently, $x^2 = 0 =y^2$ and so $\mathcal{M}_1^2 = \langle xy \rangle$. Next, we find all the nontrivial ideals of the ring $R_1$. Let $I = \langle a_1x + a_2y \rangle$,  for some $a_1, a_2 \in R_1$, be the principal ideal of $R_1$ except $\langle x \rangle$, $\langle y \rangle$ and $\langle xy \rangle$. Clearly both $a_1, a_2 \notin \mathcal{M}_1$. If $a_1 \in U(R_1)$, then $I = \langle x + a_3y \rangle$ for some $a_3 \in R_1$. Suppose that $a_3 \in \mathcal{M}_1$, then we have $I = \langle x \rangle$. Therefore, principal ideals of $R_1$ are of the form $\langle xy \rangle, \langle y \rangle, \langle x + a_3y \rangle$ where $a_3 \in U(R_1) \cup \{0\}$. Note that $\mathcal{M}_1$ is the only ideal which is not principal. Therefore, all the nontrivial ideals of $R_1$ are $\langle xy \rangle, \langle x, y \rangle, \langle y \rangle, \langle x + a_3y \rangle$, where $a_3 \in U(R_1) \cup \{0\}$. If $|R_1/ \mathcal{M}_1| \ge 3$, then there exists $a \in U(R_1) \setminus \{1\}$ and $1-a \notin \mathcal{M}_1$. Now, we show that $\langle x + y \rangle + \langle x + ay \rangle = \mathcal{M}_1$. Clearly, $\langle x + y \rangle + \langle x + ay \rangle \subseteq \mathcal{M}_1$. Since $1-a \notin \mathcal{M}_1$, we have $1-a$ is a unit of $R_1$ and so $x = (1-a)^{-1}(x + ay) + [1-(1-a)^{-1}](x +y)$. It follows that $x \in \langle x + ay \rangle + \langle x + y \rangle$. Similarly, $y \in \langle x + ay \rangle + \langle x + y \rangle$. Consequently, $\mathcal{M} \subseteq \langle x + uy \rangle + \langle x + vy \rangle$. Let $|R_1/ \mathcal{M}_1| = 2$. It follows that $\text{PIS}(R_1)$ has excatly $5$ vertices $\langle x \rangle$, $\langle y \rangle$, $\langle xy \rangle$, $\langle x, y \rangle$ and $\langle x + y \rangle$. Thus, $J_1 = \langle x\rangle \times \langle 0 \rangle$, $J_2 = \langle y \rangle \times \langle 0 \rangle$, $J_3 = \langle x+y \rangle \times \langle 0 \rangle$, $J_4 = R_1 \times \langle 0 \rangle$, $J_5 = \langle xy \rangle \times \langle 0 \rangle$, $J_6 = \langle 0 \rangle \times F_2$, $J_7 = \langle x \rangle  \times F_2$, $J_8 = \langle y \rangle  \times F_2$, $J_9 = \langle x+y \rangle  \times F_2$, $J_{10} = \mathcal{M}_1  \times F_2$, $J_{11} = \mathcal{M}_1 \times \langle 0 \rangle$, $J_{12} = \langle xy \rangle \times F_2$ are the vertices of $\text{PIS}(R_1 \times F_2)$.

Now, consider $X = \{J_1, J_3, J_4, J_7, J_8, J_{10}, J_{11}\}$. Note that the subgraph $\Omega = \text{PIS}(X) - \{(J_1, J_8), (J_7, J_8), (J_3, J_{10}), \\ (J_{10}, J_{11})\}$ is homeomorphic to $K_{3,3}$. By Proposition \ref{genus}, we have $g(\Omega) = 1$ and by Lemma \ref{eulerformulagenus}, we get three faces in any embedding of $K_{3,3}$ in $\mathbb{S}_1$. Consequently, any embedding of $\Omega$ has at least one face of length $6$. If we insert the adjacent vertices $J_2, J_9$  and the respective edges incident to $J_2$ and $J_9$ to the embedding of $\Omega$, then $(J_2, J_9)$ must be embedded in a face of length $6$. To embed the graph $\text{PIS}(R)$ in $\mathbb{S}_1$, we need to add the edges  $\{(J_1, J_8), (J_7, J_8), (J_3, J_{10}), (J_{10}, J_{11})\}$ and remaining vertices to the embedding of $K_{3,3}$. Observe that this is not possible without edge crossings. It follows that $g(\text{PIS}(R)) >1$, a contradiction. If $x^2 = 0 =y^2 = xy$, then by the similar argument, we get a contradiction. Similarly, if $|R_1/ \mathcal{M}_1| \ge 3$, we get a contradiction. Therefore, the local rings $R_1$ and $R_2$ are principal.

Let $R \cong R_1 \times R_2$, where both $R_1$ and $R_2$ are principal ideal rings. If $\eta(\mathcal{M}_1) = \eta(\mathcal{M}_2) =2$, then by Theorem \ref{Planar_primeidealsum}, we get a contradiction. Next, let $\eta(\mathcal{M}_1), \eta(\mathcal{M}_2) \ge 3$. Then $\text{PIS}(R)$ contains a subgraph with two blocks of $K_{3,3}$ (see Figure \ref{blockofk33fig2}). By Proposition \ref{genus} and Lemma \ref{genusofblocks}, we have $g(\text{PIS}(R)) >1$, a contradiction. 
\begin{figure}[h!]
\centering
\includegraphics[width=0.6 \textwidth]{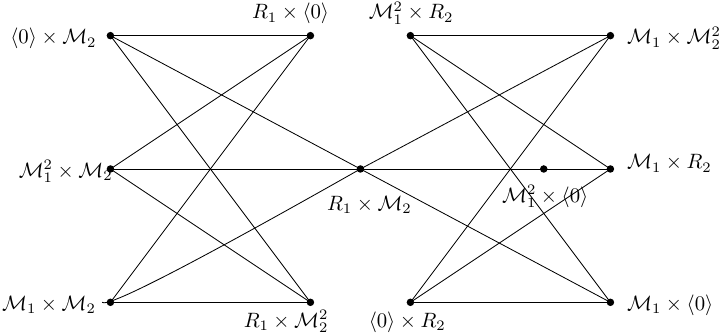}
\caption{Subgraph of $\text{PIS}(R_1 \times R_2)$, where $\eta(\mathcal{M}_1), \eta(\mathcal{M}_2) \ge 3$ }
\label{blockofk33fig2}
\end{figure}
Now assume that $R_1$ is a principal ideal ring and $R_2$ is a field. By Theorem \ref{Planar_primeidealsum}, we get $\text{PIS}(R_1 \times F_2)$ is a planar graph, a contradiction. If $\eta(\mathcal{M}_1) \ge 4$ and $\eta(\mathcal{M}_2)=2$, then by Figure \ref{blockK_33fig3} and Lemma \ref{genusofblocks}, we obtain $g(\text{PIS}(R)) \geq 2$, a contradiction. 
\begin{figure}[h!]
\centering
\includegraphics[width=0.6 \textwidth]{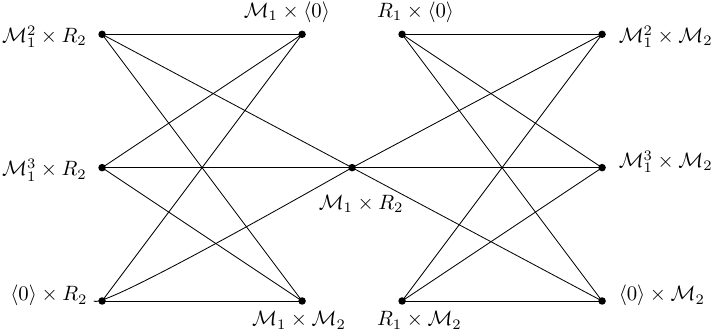}
\caption{Subgraph of $\text{PIS}(R_1 \times R_2)$, where $\eta(\mathcal{M}_1) \ge 4$ and $\eta(\mathcal{M}_2)=2$}
\label{blockK_33fig3}
\end{figure}

Therefore, $R \cong R_1 \times R_2$, where $R_1$ and $R_2$ are principal ideal rings with $\eta(\mathcal{M}_1) = 3$ and $\eta(\mathcal{M}_2) = 2$.
The converse follows from Figure \ref{genus1fig2}.
\begin{figure}[h!]
\centering
\includegraphics[width=0.6 \textwidth]{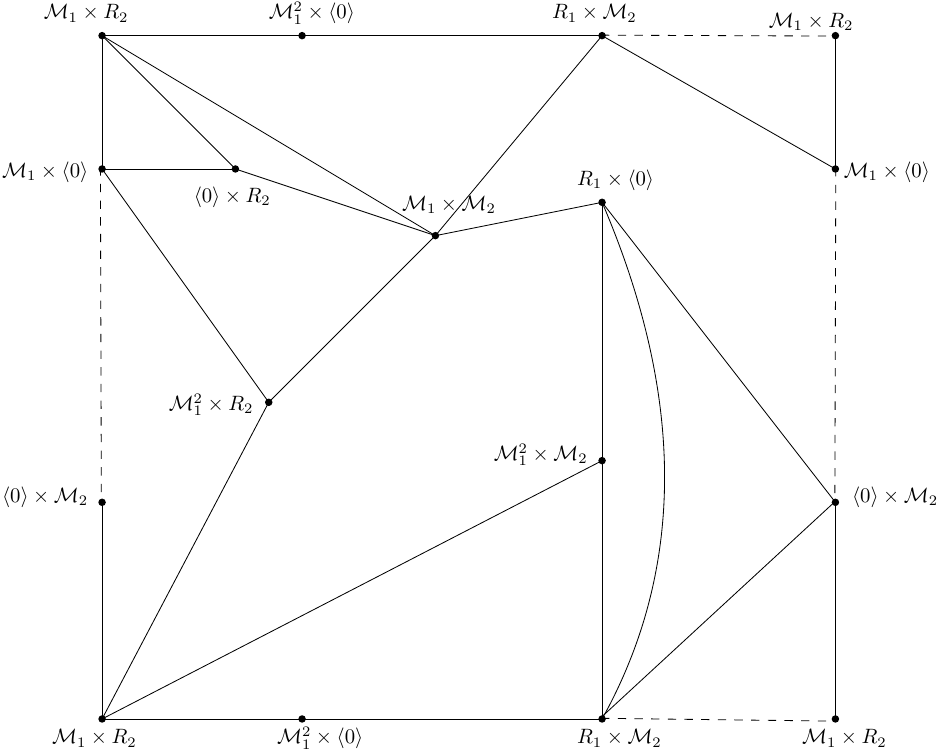}
\caption{Embedding of $\text{PIS}(R_1 \times R_2)$ in $\mathbb{S}_1$, where $\eta(\mathcal{M}_1) = 3$ and $\eta(\mathcal{M}_2) = 2$}
\label{genus1fig2}
\end{figure}
\end{proof}

\subsection{Crosscap of $\text{PIS}(R)$}
In this subsection, we classify all the non-local commutative Artinian rings $R$ with unity such that the crosscap of $\text{PIS}(R)$ is at most two. The following theorem asserts that for a non-local commutative ring $R$ its prime ideal sum graph $\text{PIS}(R)$ cannot be of crosscap one.

\begin{theorem}
    For any non-local commutative Artinian ring $R$ with unity, the prime ideal sum graph $\textnormal{PIS}(R)$ can not be projective planar.  
\end{theorem}

\begin{proof}
Let $R \cong R_1 \times R_2 \times \cdots \times R_n$ ($n \ge 2$), where each $R_i$ is local ring with maximal ideal $\mathcal{M}_i$. Suppose that $cr(\text{PIS}(R)) = 1$. Let $n \ge 5$. By the proof of Lemma \ref{genusgreaterthan2}, we get a subgraph of $\text{PIS}(R)$ which is homeomorphic to $K_{5,5}$. By Proposition \ref{crosscap}, we get $n \le 4$. Now suppose that $R \cong F_1 \times F_2 \times F_3 \times F_4$. Then we get $v= 14$ and $e = 48$. Lemma \ref{eulerformulacrosscap} yields $f = 35$. It follows that $2e < 3f$, a contradiction. Thus, $n \le 3$. We may now suppose that $R \cong R_1 \times R_2 \times R_3$. If each $R_i$ is a field, then by Theorem \ref{Planar_primeidealsum}, we get a contradiction. Next, assume that both $R_1$ and $R_2$ are not fields but $R_3$ is a field. By Figure \ref{K_54subdivisionR_1R_2F_3prime}, the graph $\text{PIS}(R_1 \times R_2 \times F_3 )$ contains a subgraph homeomorphic to $K_{5,4}$, a contradiction. Now suppose that $R_1$ is not a field but both $R_2$ and $R_3$ are fields. If $\mathcal{I}^*(R_1) = \{ K, \mathcal{M}_1 \}$, then by Figure \ref{blockK_33fig1primeideal} and Lemma \ref{crosscapofblocks}, we get $cr(\text{PIS}(R)) > 1$, a contradiction. Therefore, $R\cong R_1 \times F_2 \times F_3$, where $\mathcal{I}^*(R_1) = \{\mathcal{M}_1 \}$. Now consider the set $X = \{ \mathcal{M}_1 \times F_2 \times \langle 0 \rangle,\ \mathcal{M}_1 \times \langle 0 \rangle \times \langle 0 \rangle,\ \mathcal{M}_1 \times F_2 \times F_3,\ R_1 \times F_2 \times \langle 0 \rangle,\ \mathcal{M}_1 \times \langle 0 \rangle \times F_3,\ \langle 0 \rangle \times F_2 \times F_3,\ \langle 0 \rangle \times F_2 \times \langle 0 \rangle,\ R_1 \times \langle 0 \rangle \times F_3 \}$ and the subgraph $G = \text{PIS}(X) - \{ (\mathcal{M}_1 \times \langle 0 \rangle \times \langle 0 \rangle, \ \mathcal{M}_1 \times F_2 \times F_3), \ ( \langle 0 \rangle \times F_2 \times \langle 0 \rangle, \ \mathcal{M}_1 \times \langle 0 \rangle \times F_3), \ ( \mathcal{M}_1 \times F_2 \times \langle 0 \rangle, \ \mathcal{M}_1 \times F_2 \times F_3 ), \ (\mathcal{M}_1 \times \langle 0 \rangle \times F_3,\ \langle 0 \rangle \times F_2 \times F_3) \}$. Then $G$ is homeomorphic to $K_{3,3}$. By Lemma \ref{crosscap}, we have $cr(G) = 1$. Moreover, by Lemma \ref{eulerformulacrosscap}, we get four faces in any embedding of $G$ in $\mathbb{N}_1$. Consequently, any embedding of $G$ in $\mathbb{N}_1$ has three faces of length $4$ and one face of length $6$. To embed $\text{PIS}(R)$ from $G$ in $\mathbb{N}_1$, first we insert vertex $R_1 \times \langle 0 \rangle \times \langle 0 \rangle$ in an embedding of $G$ in $\mathbb{N}_1$. Both the adjacent vertices $R_1 \times \langle 0 \rangle \times \langle 0 \rangle$ and $\langle 0 \rangle \times \langle 0 \rangle \times F_3$ must be inserted in the same face $F'$. Note that $R_1 \times \langle 0 \rangle \times \langle 0 \rangle$ is adjacent to $\mathcal{M}_1 \times F_2 \times \langle 0 \rangle$, $R_1 \times F_2 \times \langle 0 \rangle$, $\mathcal{M}_1 \times \langle 0 \rangle \times F_3$. Also, $\langle 0 \rangle \times \langle 0 \rangle \times F_3 \sim \mathcal{M}_1 \times F_2 \times F_3$. Moreover, $\langle 0 \rangle \times \langle 0 \rangle \times F_3 \sim R_1 \times \langle 0 \rangle \times F_3 \sim M_1 \times \langle 0 \rangle \times \langle 0 \rangle$. Consequently, the face $F'$ should be of length $6$. After inserting $R_1 \times \langle 0 \rangle \times \langle 0 \rangle$ in $F'$, we need to add the edges $(\langle 0 \rangle \times \langle 0 \rangle \times F_3,\  R_1 \times \langle 0 \rangle \times \langle 0 \rangle ),\ (\langle 0 \rangle \times \langle 0 \rangle \times F_3,\ \mathcal{M}_1 \times F_2 \times F_3), \ (\langle 0 \rangle \times \langle 0 \rangle \times F_3,\ R_1 \times \langle 0 \rangle \times F_3), (\langle 0 \rangle \times \langle 0 \rangle \times F_3, \ \mathcal{M}_1 \times F_2 \times \langle 0 \rangle)$, which is not possible without edge crossing (see Figure \ref{faceF'}), a contradiction. Thus, $R \cong R_1 \times R_2$.

\begin{figure}[h!]
\centering
\includegraphics[width=0.4 \textwidth]{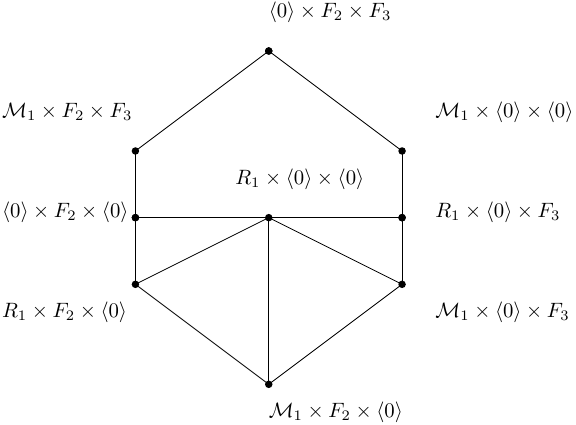}
\caption{The face $F'$ of $\text{PIS}(R_1 \times F_2 \times F_3)$}
\label{faceF'}
\end{figure}

Now let $R \cong R_1 \times R_2$. Assume that one of $R_i$ $(1 \le i \le 2)$ is not a principal ideal ring. Without loss of generality, let $R_1$ is not a principal ideal ring. In the similar lines of the proof of Theorem \ref{genusofR1R2}, we get $cr(\text{PIS}(R)) >1$. Therefore, $ R \cong R_1 \times R_2$, where the local rings $R_1$ and $R_2$ are principal. If both $R_1$ and $R_2$ are fields, then the graph $\text{PIS}(R_1 \times R_2)$ is planar, a contradiction. If $\eta(\mathcal{M}_1) = \eta(\mathcal{M}_2) = 2$, then by Theorem \ref{Planar_primeidealsum}, the graph $\text{PIS}(R_1 \times R_2)$ is planar, again a contradiction. Next, let $\eta(\mathcal{M}_1), \eta(\mathcal{M}_2) \ge 3$. Then $\text{PIS}(R)$ contains a subgraph which has two $K_{3,3}$ blocks (see Figure \ref{blockofk33fig2}). By Proposition \ref{crosscap} and Lemma \ref{crosscapofblocks}, we get $cr(\text{PIS}(R)) >1$, a contradiction. Now assume that $R_1$ is a principal ideal ring and $R_2$ is a field. By Theorem, \ref{Planar_primeidealsum}, the graph $\text{PIS}(R)$ is planar, a contradiction. Next, let $\eta(\mathcal{M}_1) \ge 4$ and $\eta(\mathcal{M}_2) = 2$. Then by Figure \ref{blockK_33fig3} and Lemma \ref{crosscapofblocks}, we obtain $cr(\text{PIS}(R)) \geq 2$, a contradiction. Now, let $R \cong R_1 \times R_2$, where $R_1$, $R_2$ are principal ideal rings with $\eta(\mathcal{M}_1) = 3$ and $\eta(\mathcal{M}_2) = 2$. Then consider the set
\[ Y = \{ R_1 \times \langle 0 \rangle, \ \mathcal{M}_1 \times R_2, \ R_1 \times \mathcal{M}_2, \ \mathcal{M}_1^2 \times \mathcal{M}_2, \ \langle 0 \rangle \times \mathcal{M}_2, \ \mathcal{M}_1 \times \mathcal{M}_2 \}  \] 
and the subgraph $G' = \text{PIS}(Y) - \{ (R_1 \times \langle 0 \rangle, R_1 \times \mathcal{M}_2 ) \} $. Then the subgraph $G'$ is homeomorphic to $K_{3,3}$. It follows that any embedding of $G'$ in $\mathbb{N}_1$ has $4$ faces. Now to embed $\text{PIS}(R)$ in $\mathbb{N}_1$ through $G'$, first we insert the vertex $\mathcal{M}_1 \times \langle 0 \rangle$. Since $\mathcal{M}_1 \times \langle 0 \rangle$ is adjacent to $\langle 0 \rangle \times R_2$ and $\mathcal{M}_1^2 \times R_2$, we must have $\mathcal{M}_1 \times \langle 0 \rangle$, $\langle 0 \rangle \times R_2$ and $\mathcal{M}_1^2 \times R_2$ in the same face $F''$. Note that $\mathcal{M}_1 \times \langle 0 \rangle$ is adjacent to $\mathcal{M}_1 \times R_2$, $R_1 \times \mathcal{M}_2$ and $\langle 0 \rangle \times R_2$. Also $\mathcal{M}_1^2 \times R_2$ is adjacent to $\mathcal{M}_1 \times \mathcal{M}_2$. Therefore, the face $F''$ contains the vertices $\mathcal{M}_1 \times R_2$, $R_1 \times \mathcal{M}_2$ and $\mathcal{M}_1 \times \mathcal{M}_2$. Now after inserting the edges $( \mathcal{M}_1 \times \langle 0 \rangle, \ \mathcal{M}_1 \times R_2)$, $( \mathcal{M}_1 \times \langle 0 \rangle, \ R_1 \times \mathcal{M}_2)$, $(\langle 0 \rangle \times R_2, \ \mathcal{M}_1 \times \mathcal{M}_2)$, $(\langle 0 \rangle \times R_2, \ \mathcal{M}_1 \times \langle 0 \rangle)$ and $(\langle 0 \rangle \times R_2, \ \mathcal{M}_1 \times R_2)$, the insertion of the edges $(\mathcal{M}_1^2 \times R_2, \ \mathcal{M}_1 \times \mathcal{M}_2)$, $(\mathcal{M}_1^2 \times R_2, \ \mathcal{M}_1 \times \langle 0 \rangle )$, $(\mathcal{M}_1^2 \times R_2, \ \mathcal{M}_1 \times R_2)$ without edge crossing is not possible. Hence, $cr(\text{PIS}(R)) \ge 2$.
\end{proof}

\begin{theorem}
    Let $R$ be a non-local commutative ring such that  $R \cong R_1 \times R_2 \times \cdots \times R_n$ $(n \geq 2)$, where each $R_i$ is a local ring with maximal ideal $\mathcal{M}_i$. Then $cr(\textnormal{PIS}(R)) = 2$ if and only if 
    \begin{enumerate}[\rm(i)]
    \item $R \cong R_1 \times F_2 \times F_3$, where $R_1$ is a principal ideal ring with $\eta(\mathcal{M}_1) =2$ and $F_2$, $F_3$ are fields.
    
    \item $R \cong R_1 \times R_2$, where both $R_1$ and $R_2$ are principal ideal rings with $\eta(\mathcal{M}_1) =3$ and $\eta(\mathcal{M}_2) =2$.
    \end{enumerate}
\end{theorem}

\begin{proof}
Suppose that $cr(\text{PIS}(R)) = 2$. If $n \ge 5$, then by the proof of the Lemma \ref{genusgreaterthan2}, we get a subgraph of $\text{PIS}(R)$ homeomorphic to $K_{5,5}$. By the Proposition \ref{crosscap}, we get $n \le 4$. Now let $R \cong F_1 \times F_2 \times F_3 \times F_4$. Then we get $v= 14$ and $e = 48$. By Lemma \ref{eulerformulacrosscap}, we have $f = 34$. It implies that $2e < 3f$, a contradiction. Thus, $n \le 3$. Let $R \cong R_1 \times R_2 \times R_3$. Suppose that both $R_1$ and $R_2$ are not fields but $R_3$ is a field. Then by Figure \ref{K_54subdivisionR_1R_2F_3prime}, $\text{PIS}(R)$ contains a subgraph homeomorphic to $K_{5,4}$, a contradiction. Next, let $R \cong R_1 \times F_2 \times F_3$ such that $\mathcal{I}^*(R_1) = \{ \mathcal{M}_1, K\}$. Then by Figure \ref{blockK_33fig1primeideal}, note that $\text{PIS}(R)$ contains a subgraph $G$, which has two $K_{3,3}$ blocks. Then $cr(\text{PIS}(G)) =2$. To embed $\text{PIS}(R)$ in $\mathbb{N}_2$ from $G$, first we insert the edges $e_1 = (\langle 0 \rangle \times F_2 \times F_3,\ \mathcal{M}_1 \times \langle 0 \rangle \times F_3)$, $e_2 = (K \times F_2 \times F_3,\ \mathcal{M}_1 \times \langle 0 \rangle \times F_3)$ and $e_3 = (\mathcal{M}_1 \times F_2 \times F_3,\ \mathcal{M}_1 \times F_2 \times \langle 0 \rangle)$. After adding the edges $e_1$ and $e_2$, note that we have two triangles: \rm(i) $\langle 0 \rangle \times F_2 \times F_3 \sim \mathcal{M}_1 \times F_2 \times F_3 \sim \mathcal{M}_1 \times \langle 0 \rangle \times F_3 \sim \langle 0 \rangle \times F_2 \times F_3$, and \rm(ii) $K \times F_2 \times F_3 \sim \mathcal{M}_1 \times F_2 \times F_3 \sim \mathcal{M}_1 \times \langle 0 \rangle \times F_3 \sim K \times F_2 \times F_3$. Also $\mathcal{M}_1 \times F_2 \times \langle 0 \rangle$ is adjacent to $\langle 0 \rangle \times F_2 \times F_3$, $\mathcal{M}_1 \times \langle 0 \rangle \times F_3$ and $K \times F_2 \times F_3$. Since the vertex $\mathcal{M}_1 \times \langle 0 \rangle \times \langle 0 \rangle$ is adjacent to $\langle 0 \rangle \times F_2 \times F_3$ and $K \times F_2 \times F_3$, we cannot insert the edge $e_3$ without edge crossing. It implies that $|\mathcal{I}^*(R_1)|=1$. Note that if $R_1$ is not a principal ideal ring, then $|\mathcal{I}^*(R_1)| \ge 2$. Therefore, $R \cong R_1 \times F_2 \times F_3$ such that $R_1$ is a principal ideal ring with $\eta( \mathcal{M}_1)=2$.
 
Next, let $R \cong R_1 \times R_2$. First, assume that one of $R_i$ $(1 \le i \le 2)$ is not a principal ideal ring. Without loss of generality, let $R_1$ is not a principal ideal ring. Then $\mathcal{M}_1 = \langle x_1, x_2, \ldots, x_r \rangle$, where $r \ge 2$. Let $r=2$ i.e., $\mathcal{M}_1 = \langle x, y \rangle$ for some $x,y \in R_1$. Now let $\mathcal{I}^*(R_2) = \{\mathcal{M}_2\}$. By the proof of Theorem \ref{genusofR1R2}, $\text{PIS}(R)$ contains a subgraph homeomorphic to $K_{5,4}$, a contradiction. Let $R \cong R_1 \times F_2$ and $\mathcal{M}_1 = \langle x,y \rangle $. If $x^2 \neq 0$, then by the proof of Theorem \ref{genusofR1R2}, $\text{PIS}(R)$ contains a subgraph homeomorphic to $K_{5,4}$, which is not possible. Therefore, $R \cong R_1 \times F_2$ and $\mathcal{M}_1 = \langle x,y \rangle $ such that $x^2 = y^2 = 0$. Let $J_1 = \langle x\rangle \times \langle 0 \rangle$, $J_2 = \langle y \rangle \times \langle 0 \rangle$, $J_3 = \langle x+y \rangle \times \langle 0 \rangle$, $J_4 = R_1 \times \langle 0 \rangle$, $J_5 = \langle xy \rangle \times \langle 0 \rangle$, $J_6 = \langle 0 \rangle \times F_2$, $J_7 = \langle x \rangle  \times F_2$, $J_8 = \langle y \rangle  \times F_2$, $J_9 = \langle x+y \rangle  \times F_2$, $J_{10} = \mathcal{M}_1  \times F_2$, $J_{11} = \mathcal{M}_1 \times \langle 0 \rangle$, $J_{12} = \langle xy \rangle \times F_2$ be the vertices of $\text{PIS}(R_1 \times F_2)$. Consider the set $X =\{J_1, J_7, J_8, J_9, J_{10}, J_{11} \}$. Then the subgraph induced the set $X \setminus \{ J_1 \}$ is isomorphic to $K_5$. Thus, $cr(\text{PIS}(X)) \ge 1$. Observe that the subgraph $\text{PIS}(X)$ is isomorphic to a subgraph of $K_6$ and all the faces of $K_6$ in $\mathbb{N}_1$ are triangles. Thus, $cr(\text{PIS}(X)) = 1$. Note that $\text{PIS}(X)$ has $6$ vertices and $13$ edges. By Lemma \ref{eulerformulacrosscap}, it follows that $f = 8$ in any embedding of $\text{PIS}(X)$ in $\mathbb{N}_1$. Therefore, $\text{PIS}(X)$ has either seven faces of length $3$ and one face of length $5$, or six faces of length $3$ and two faces of length $4$. Notice that the $5$-length face of $\text{PIS}(X)$ contains the vertices $J_1, J_7, J_8, J_{10}$ and $J_{11}$. Now first we insert $J_2$ and $J_4$ in an embedding of $\text{PIS}(X)$ in $\mathbb{N}_1$. Since $J_2 \sim J_4$, they must be inserted in the same face. Note that $J_2$ is adjacent to $J_7, J_9, J_{10}$ and $J_4 \sim J_{11}$. Also, $J_{11}$ is adjacent to $J_7, J_8, J_9$ and $J_{10}$. Moreover, $J_9 \sim J_{10}$, $J_7 \sim J_9$ and $J_ 7\sim J_{10}$. It follows that both $J_2$ and $J_4$ must be inserted in a face of length $5$ containing the vertices $J_7, J_9, J_{10}, J_{11}$ which is not possible. Consequently, $ cr(\text{PIS}(X \cup \{ J_2, J_4\})) \ge 2$. Next, consider the set $Y = V(\text{PIS}(R)) \setminus \{J_3\}$. Since $(X \cup \{ J_2, J_4\}) \subsetneq Y$, we have $cr(\text{PIS}(Y)) \ge 2 $. For the subgraph $\text{PIS}(Y)$, note that $v=9$ and $e = 21$. It implies that $f=12$ in any embedding of $\text{PIS}(Y)$ in $\mathbb{N}_2$. Thus, $\text{PIS}(Y)$ have seven faces of length $3$, two faces of length $4$ and one face of length $5$ in any embedding of $\text{PIS}(Y)$ in $\mathbb{N}_2$. Now to embed $\text{PIS}(R)$ in $\mathbb{N}_2$, we insert $J_3$ in a face $F$ of an embedding of $\text{PIS}(Y)$. Then $F$ must contain $J_4$, $J_7$, $J_8$ and $J_{10}$, but $J_4$ is not adjacent to $J_7$, $J_8$ and $J_{10}$. Thus, $F$ must contain at least $6$ vertices, which is not possible. It follows that $cr(\text{PIS}(R)) \ge 3$.   

Consequently, $ R \cong R_1 \times R_2$, where $R_1$ and $R_2$ are principal local rings. If both $R_1$ and $R_2$ have at most one nontrivial ideal i.e., $\eta(\mathcal{M}_1), \eta(\mathcal{M}_2) \le 2$, then by Theorem \ref{Planar_primeidealsum}, we get a contradiction. Next, let $\eta(\mathcal{M}_1), \eta(\mathcal{M}_2) \ge 3$. Note that $\text{PIS}(R)$ contains a subgraph $G_1$ as two blocks of $K_{3,3}$ (see Figure \ref{blockofk33fig2}). By Lemma \ref{crosscapofblocks}, $cr(\text{PIS}(G_1)) =2$. To embed $\text{PIS}(R)$ in $\mathbb{N}_2$ from $G_1$, first we insert the edges $e_1' = (\langle 0 \rangle \times R_2 ,\ \mathcal{M}_1 \times \mathcal{M}_2 )$, $e_2' = (\mathcal{M}_1^2 \times R_2, \ \mathcal{M}_1 \times \mathcal{M}_2)$, $e_3' = (\mathcal{M}_1 \times R_2 ,\ \mathcal{M}_1 \times \mathcal{M}_2)$ and $e_4' = (\mathcal{M}_1 \times R_2 ,\ \mathcal{M}_1 \times \langle 0 \rangle)$. After inserting the edges $e_1'$, $e_2'$ and $e_3'$, note that we have two triangles: \rm(i) $\langle 0 \rangle \times R_2 \sim \mathcal{M}_1 \times \mathcal{M}_2 \sim \mathcal{M}_1 \times R_2 \sim \langle 0 \rangle \times R_2$, and \rm (ii) $\mathcal{M}_1 \times \mathcal{M}_2  \sim \mathcal{M}_1^2 \times R_2  \sim \mathcal{M}_1 \times R_2 \sim \mathcal{M}_1 \times \mathcal{M}_2$. Also, $\mathcal{M}_1 \times \langle 0 \rangle$ is adjacent to $\langle 0 \rangle \times R_2 $, $\mathcal{M}_1^2 \times R_2 $, $R_1 \times \mathcal{M}_2$. Moreover, $R_1 \times \mathcal{M}_2 \sim \mathcal{M}_1 \times \mathcal{M}_2$. Since $\langle 0 \rangle \times R_2 \sim \mathcal{M}_1 \times \mathcal{M}_2^2 \sim \mathcal{M}_1^2 \times R_2$, we cannot insert the edge $e_4'$ without edge crossing. Now, let $\eta(\mathcal{M}_1) \ge 4$ and $\eta(\mathcal{M}_2) = 2$. Then $\text{PIS}(R)$ contains a subgraph $G''$ which has two blocks of $K_{3,3}$ (see Figure \ref{blockK_33fig3}). Then by Lemma \ref{crosscapofblocks}, $cr(\text{PIS}(G'')) = 2$. To embed $\text{PIS}(R)$ in $\mathbb{N}_2$ through $G''$, first we insert the edges $e_1'' = ( R_1 \times \mathcal{M}_2,\ \mathcal{M}_1 \times \mathcal{M}_2)$, $e_2'' = ( \mathcal{M}_1 \times \langle 0 \rangle,\ \mathcal{M}_1 \times R_2)$, $e_3'' = ( R_1 \times \mathcal{M}_2,\ \mathcal{M}_1^2 \times \langle 0 \rangle)$, $e_4'' = ( \mathcal{M}_1^2 \times \langle 0 \rangle,\ \mathcal{M}_1 \times R_2)$ and $e_5'' = (R_1 \times \mathcal{M}_2,\ \mathcal{M}_1 \times \langle 0 \rangle )$. Note that $\mathcal{M}_1 \times \mathcal{M}_2 \sim  \langle 0 \rangle \times R_2 \sim  \mathcal{M}_1 \times \langle 0 \rangle \sim \mathcal{M}_1^2 \times R_2 \sim  \mathcal{M}_1 \times \mathcal{M}_2$. Also, $\langle 0 \rangle \times R_2 \sim \mathcal{M}_1 \times R_2 \sim \mathcal{M}_1^2 \times R_2$. Then after inserting the edges $e_1''$, $e_2''$, $e_3''$ and $e_4''$, we can not insert the edge $e_5''$ without edge crossing in an embedding of $G''$ in $\mathbb{N}_2$. Hence, $R \cong R_1 \times R_2$, where $R_1$, $R_2$ are principal ideal rings with $\eta(\mathcal{M}_1) = 3$ and $\eta(\mathcal{M}_2) = 2$. 

Converse follows from Figure \ref{crosscapfig1} and Figure \ref{crosscapfig2}.
\begin{figure}[h!]
\centering
\includegraphics[width=0.7 \textwidth]{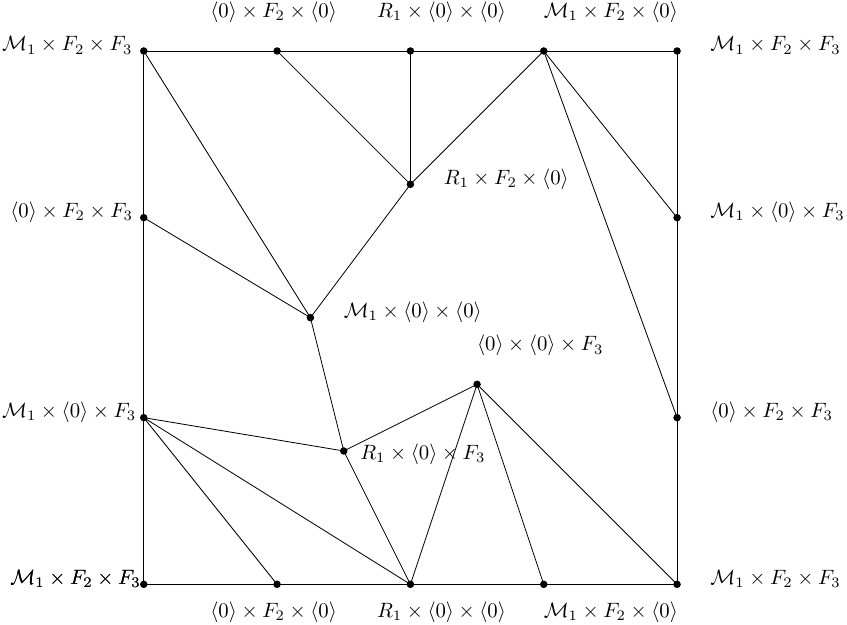}
\caption{Embedding of $\text{PIS}(R_1 \times F_2 \times F_3)$ in $\mathbb{N}_2$, where $\eta(\mathcal{M}_1) = 2$}
\label{crosscapfig1}
\end{figure}
\begin{figure}[h!]
\centering
\includegraphics[width=0.6 \textwidth]{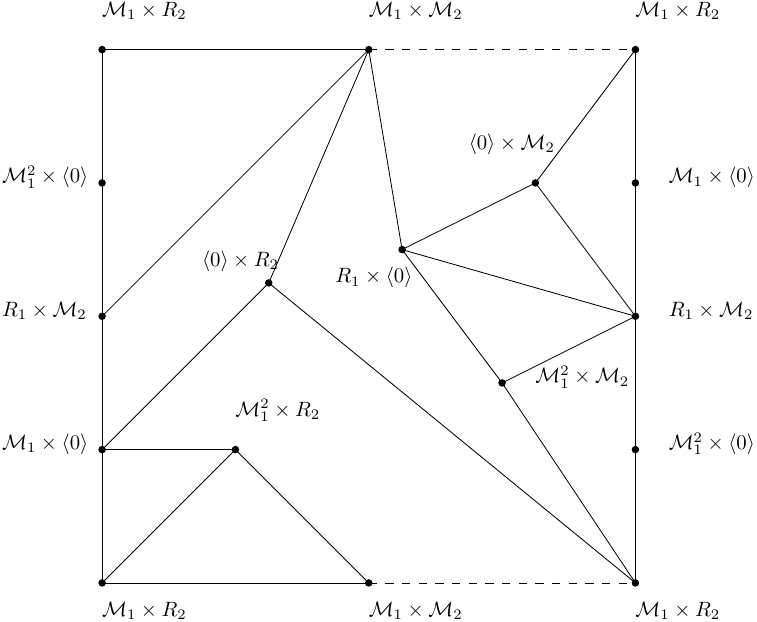}
\caption{Embedding of $\text{PIS}(R_1 \times R_2)$ in $\mathbb{N}_2$, where $\eta(\mathcal{M}_1) = 3$ and $\eta(\mathcal{M}_2) = 2$}
\label{crosscapfig2}
\end{figure}
\end{proof}

\vspace{.5cm}
\textbf{Acknowledgement:} We would like to thank the referees for their valuable suggestions which helped us to improve the presentation of the paper.

\section*{Declarations}

\textbf{Funding}: The first author gratefully acknowledges Birla Institute of Technology and Science (BITS) Pilani, India, for providing financial support.

\vspace{.3cm}
\textbf{Conflicts of interest/Competing interests}: There is no conflict of interest regarding the publishing of this paper. 

\vspace{.3cm}
\textbf{Availability of data and material (data transparency)}: Not applicable.

\vspace{.3cm}
\textbf{Code availability (software application or custom code)}: Not applicable.


\begin{thebibliography}{10}

\bibitem{afkhami2015planar}
M.~Afkhami, M.~Farrokhi D.~G., and K.~Khashyarmanesh.
\newblock Planar, toroidal, and projective commuting and noncommuting graphs.
\newblock {\em Comm. Algebra}, 43(7):2964--2970, 2015.

\bibitem{afkhami2012generalized}
M.~Afkhami, K.~Khashyarmanesh, and K.~Nafar.
\newblock Generalized {C}ayley graphs associated to commutative rings.
\newblock {\em Linear Algebra Appl.}, 437(3):1040--1049, 2012.

\bibitem{akbari2015inclusion}
S.~Akbari, M.~Habibi, A.~Majidinya, and R.~Manaviyat.
\newblock The inclusion ideal graph of rings.
\newblock {\em Comm. Algebra}, 43(6):2457--2465, 2015.

\bibitem{akbari2009total}
S.~Akbari, D.~Kiani, F.~Mohammadi, and S.~Moradi.
\newblock The total graph and regular graph of a commutative ring.
\newblock {\em J. Pure Appl. Algebra}, 213(12):2224--2228, 2009.

\bibitem{akbari2003zero}
S.~Akbari, H.~R. Maimani, and S.~Yassemi.
\newblock When a zero-divisor graph is planar or a complete {$r$}-partite
  graph.
\newblock {\em J. Algebra}, 270(1):169--180, 2003.

\bibitem{anderson2008zero}
D.~F. Anderson and A.~Badawi.
\newblock On the zero-divisor graph of a ring.
\newblock {\em Comm. Algebra}, 36(8):3073--3092, 2008.

\bibitem{anderson2008total}
D.~F. Anderson and A.~Badawi.
\newblock The total graph of a commutative ring.
\newblock {\em J. Algebra}, 320(7):2706--2719, 2008.

\bibitem{anderson1999zero}
D.~F. Anderson and P.~S. Livingston.
\newblock The zero-divisor graph of a commutative ring.
\newblock {\em J. Algebra}, 217(2):434--447, 1999.

\bibitem{anitha2020characterization}
T.~Anitha and R.~Rajkumar.
\newblock Characterization of groups with planar, toroidal or projective planar
  (proper) reduced power graphs.
\newblock {\em J. Algebra Appl.}, 19(5):2050099, 2020.

\bibitem{asir2018classification}
T.~Asir and K.~Mano.
\newblock The classification of rings with its genus of class of graphs.
\newblock {\em Turkish J. Math.}, 42(3):1424--1435, 2018.

\bibitem{asir2019classification}
T.~Asir and K.~Mano.
\newblock Classification of rings with crosscap two class of graphs.
\newblock {\em Discrete Appl. Math.}, 265:13--21, 2019.

\bibitem{asir2020classification}
T.~Asir and K.~Mano.
\newblock Classification of non-local rings with genus two zero-divisor graphs.
\newblock {\em Soft Comput}, 24(1):237--245, 2020.

\bibitem{atiyah1994introduction}
M.~Atiyah.
\newblock {\em Introduction to Commutative Algebra}.
\newblock Addison-Wesley Publishing Company, 1994.

\bibitem{beck1988coloring}
I.~Beck.
\newblock Coloring of commutative rings.
\newblock {\em J. Algebra}, 116(1):208--226, 1988.

\bibitem{behboodi2011annihilating}
M.~Behboodi and Z.~Rakeei.
\newblock The annihilating-ideal graph of commutative rings {II}.
\newblock {\em J. Algebra Appl.}, 10(4):741--753, 2011.

\bibitem{belshoff2007planar}
R.~Belshoff and J.~Chapman.
\newblock Planar zero-divisor graphs.
\newblock {\em J. Algebra}, 316(1):471--480, 2007.

\bibitem{biswas2022subgraph}
B.~Biswas, S.~Kar, and M.~K. Sen.
\newblock Subgraph of generalized co-maximal graph of commutative rings.
\newblock {\em Soft Comput, 26:1587–1596}, 2022.

\bibitem{chakrabarty2009intersection}
I.~Chakrabarty, S.~Ghosh, T.~K. Mukherjee, and M.~K. Sen.
\newblock Intersection graphs of ideals of rings.
\newblock {\em Discrete Math.}, 309(17):5381--5392, 2009.

\bibitem{das2015genus}
A.~K. Das and D.~Nongsiang.
\newblock On the genus of the nilpotent graphs of finite groups.
\newblock {\em Comm. Algebra}, 43(12):5282--5290, 2015.

\bibitem{das2016genus}
A.~K. Das and D.~Nongsiang.
\newblock On the genus of the commuting graphs of finite non-abelian groups.
\newblock {\em Int. Electron. J. Algebra}, 19:91--109, 2016.

\bibitem{kalaimurugan2021genus}
G.~Kalaimurugan, S.~Gopinath, and T.~Tamizh~Chelvam.
\newblock On the genus of non-zero component union graphs of vector spaces.
\newblock {\em Hacet. J. Math. Stat.}, 50(6):1595--1608, 2021.

\bibitem{kalaimurugan2022genus}
G.~Kalaimurugan, P.~Vignesh, and T.~Tamizh~Chelvam.
\newblock Genus two nilpotent graphs of finite commutative rings.
\newblock {\em J. Algebra Appl.}, 22(6):2350123, 2023.

\bibitem{khashyarmanesh2013projective}
K.~Khashyarmanesh and M.~R. Khorsandi.
\newblock Projective total graphs of commutative rings.
\newblock {\em Rocky Mountain J. Math.}, 43(4):1207--1213, 2013.

\bibitem{khorsandi2022nonorientable}
M.~R. Khorsandi and S.~R. Musawi.
\newblock On the nonorientable genus of the generalized unit and unitary
  {C}ayley graphs of a commutative ring.
\newblock {\em Algebra Colloq.}, 29(1):167--180, 2022.

\bibitem{ma2020finite}
X.~Ma and H.~Su.
\newblock Finite groups whose noncyclic graphs have positive genus.
\newblock {\em Acta Math. Hungar.}, 162(2):618--632, 2020.

\bibitem{ma2016automorphism}
X.~Ma and D.~Wong.
\newblock Automorphism group of an ideal-relation graph over a matrix ring.
\newblock {\em Linear Multilinear Algebra}, 64(2):309--320, 2016.

\bibitem{maimani2008comaximal}
H.~R. Maimani, M.~Salimi, A.~Sattari, and S.~Yassemi.
\newblock Comaximal graph of commutative rings.
\newblock {\em J. Algebra}, 319(4):1801--1808, 2008.

\bibitem{maimani2012rings}
H.~R. Maimani, C.~Wickham, and S.~Yassemi.
\newblock Rings whose total graphs have genus at most one.
\newblock {\em Rocky Mountain J. Math.}, 42(5):1551--1560, 2012.

\bibitem{mohar2001graphs}
B.~Mohar and C.~Thomassen.
\newblock {\em Graphs on surfaces}.
\newblock Johns Hopkins University Press, 2001.

\bibitem{pucanovic2014toroidality}
Z.~S. Pucanovi\'{c} and Z.~Z. Petrovi\'{c}.
\newblock Toroidality of intersection graphs of ideals of commutative rings.
\newblock {\em Graphs Combin.}, 30(3):707--716, 2014.

\bibitem{pucanovic2014genus}
Z.~S. Pucanovi\'{c}, M.~Radovanovi\'{c}, and A.~L. Eri\'{c}.
\newblock On the genus of the intersection graph of ideals of a commutative
  ring.
\newblock {\em J. Algebra Appl.}, 13(5):1350155, 2014.

\bibitem{ramanathan2021projective}
V.~Ramanathan.
\newblock On projective intersection graph of ideals of commutative rings.
\newblock {\em J. Algebra Appl.}, 20(2):2150017, 2021.

\bibitem{rilwan2020genus}
N.~M. Rilwan and S.~V. Devi.
\newblock On genus of $k$-subspace intersection graph of vector space.
\newblock In {\em AIP Conference Proceedings, 2261, 030139}, 2020.

\bibitem{saha2022prime}
M.~Saha, A.~Das, E.~Y. \c{C}elikel, and C.~Abdio\u{g}lu.
\newblock Prime ideal sum graph of a commutative ring.
\newblock {\em J. Algebra Appl.}, 22(6):2350121, 2023.

\bibitem{selvakumar2022genus}
K.~Selvakumar, V.~Ramanathan, and C.~Selvaraj.
\newblock On the genus of dot product graph of a commutative ring.
\newblock {\em Indian J. Pure Appl. Math.}, 54(2):558--567, 2023.

\bibitem{chelvam2013genus}
T.~Tamizh~Chelvam and T.~Asir.
\newblock On the genus of the total graph of a commutative ring.
\newblock {\em Comm. Algebra}, 41(1):142--153, 2013.

\bibitem{tamizh2020genus}
T.~Tamizh~Chelvam and K.~Prabha~Ananthi.
\newblock The genus of graphs associated with vector spaces.
\newblock {\em J. Algebra Appl.}, 19(5):2050086, 2020.

\bibitem{wang2006zero}
H.-J. Wang.
\newblock Zero-divisor graphs of genus one.
\newblock {\em J. Algebra}, 304(2):666--678, 2006.

\bibitem{westgraph}
D.~B. West.
\newblock {\em Introduction to Graph Theory, 2nd edn.}
\newblock (Prentice Hall), 1996.

\bibitem{white1985graphs}
A.~T. White.
\newblock {\em Graphs, groups and surfaces}.
\newblock Elsevier, 1985.

\bibitem{white2001graphs}
A.~T. White.
\newblock {\em Graphs of groups on surfaces: interactions and models}.
\newblock Elsevier, 2001.

\bibitem{wickham2008classification}
C.~Wickham.
\newblock Classification of rings with genus one zero-divisor graphs.
\newblock {\em Comm. Algebra}, 36(2):325--345, 2008.

\bibitem{wickham2009rings}
C.~Wickham.
\newblock Rings whose zero-divisor graphs have positive genus.
\newblock {\em J. Algebra}, 321(2):377--383, 2009.

\bibitem{ye2012co}
M.~Ye and T.~Wu.
\newblock Co-maximal ideal graphs of commutative rings.
\newblock {\em J. Algebra Appl.}, 11(6):1250114, 2012.

\end{thebibliography}

 \end{document}